\documentclass[twoside,12pt]{article}
\usepackage{graphicx,amsmath,latexsym,amssymb,amsthm}
\usepackage[colorlinks=black,urlcolor=black]{hyperref}
at 9pt  

\date{}
\newtheorem{theorem}{Theorem}[section]

\numberwithin{equation}{section}

\begin{document}
\setlength{\unitlength}{1cm}



\vskip1.5cm

 \centerline { \textbf{Boundary value problem with transmission conditions  }}
 \centerline { \textbf{}}

\vskip.2cm


\vskip.8cm \centerline {\textbf{K. Aydemir$^*$ and O. Sh.
Mukhtarov$^*$ }}

\vskip.5cm

\centerline {$^*$Department of Mathematics, Faculty of Science,}
\centerline {Gaziosmanpa\c{s}a University,
 60250 Tokat, Turkey}
\centerline {e-mail : {\tt omukhtarov@yahoo.com,
kadriye.aydemir@gop.edu.tr }}


\vskip.5cm \hskip-.5cm{\small{\bf Abstract :} One important
innovation here is that for the Sturm-Liouville considered equation
together with eigenparameter dependent boundary conditions and two
supplementary transmission conditions at one interior point. We
develop Green's function method for spectral analysis of the
considered problem in modified Hilbert space.


\section{\textbf{Introduction}}
In physics many problems arise in the form of boundary value
problems involving second order ordinary differential equations.
This derivation based on that in \cite{kr}, however, is more
thorough than that in most elementary physics texts; while most
parameters such as density and other thermal properties are treated
as constant in such treatments, the following allows fundamental
properties of the bar to vary as a function of the bar's length,
which will lead to a Sturm-Liouville problem of a more general
nature.  In this study we shall investigate one discontinuous
eigenvalue problem which consists of Sturm-Liouville equation,
\begin{equation}\label{1.1}
\Gamma (y):=-y^{\prime \prime }(x,\lambda)+ q(x)y(x,\lambda)=\lambda
y(x,\lambda)
\end{equation}
to hold in finite interval $(-\pi, \pi) $ except at one inner point
$0 \in (-\pi, \pi) $ , where discontinuity in $u  \ \textrm{and} \
u'$ are prescribed by transmission conditions
\begin{equation}\label{1.2}
\Gamma_{1}(y):=a_{1}y'(0-,\lambda)+a_{2}y(0-,\lambda)+a_{3}y'(0+,\lambda)+a_{4}y(0+,\lambda)=0,
\end{equation}
\begin{equation}\label{1.3}
\Gamma_{2}(y):=b_{1}y'(0-,\lambda)+b_{2}y(0-,\lambda)+b_{3}y'(0+,\lambda)+b_{4}y(0+,\lambda)=0,
\end{equation}
together with the  boundary conditions
\begin{equation}\label{1.4}
 \Gamma_{3}(y):=\cos \alpha y(-\pi,\lambda)+\sin\alpha y'(-\pi,\lambda)=0,
\end{equation}
\begin{equation}\label{1.5}
\Gamma_{4}(y):=\cos\beta y(\pi,\lambda)+\sin\beta y'(\pi,\lambda)=0,
\end{equation}
where the potential  $q(x)$ is real-valued, continuous in each
interval $[-1, 0) \\  \textrm{and} \ (0, 1]$ and has a finite limits
$q(\mp c)$ ; $\alpha_{0}, \beta_{0}, \beta_{1}$ are real numbers; \
$\lambda$ \ is a complex eigenparameter.  In this study by using an
own technique we introduce a new equivalent inner product in the
Hilbert space $L_{2}(-1,0)\oplus L_{2}(0,1)$ and a linear operator
in it such a way that the considered problem can be interpreted as
eigenvalue problem for this operator.

\section{ Preliminary Results
 }
Let T=$ \left[%
\begin{array}{cccc}
  a_{1} & a_{2} & a_{3} & a_{4} \\
  b_{1} & b_{2} & b_{3} & b_{4}
  \\
\end{array} %
 \right]. $
 Denote  the determinant of the i-th
 and
j-th columns of the matrix T  by $\rho_{ij}$. Note that throughout
this study we shall assume that $ \rho_{12}>0 \ \ \textrm{and} \
\rho_{34}>0.$

In this section we shall define two basic solutions $\phi(x,\lambda)
\ \textrm{and} \ \chi(x,\lambda)$ by own technique as follows.  At
first, let us consider solutions of the equation (\ref{1.1})  on the
left-hand $\left[ -\pi,0\right)$ of the considered interval
$\left[\pi,0\right)\cup(0,\pi]$ satisfying the initial conditions
\begin{eqnarray}\label{(2.10)}
\label{tam1}
 \ \ y(-\pi,\lambda)=\sin\alpha , \
\ \frac{\partial y(-\pi,\lambda)}{\partial x}=-\cos\alpha
\end{eqnarray}
By virtue of well-known existence and uniqueness theorem of ordinary
differential equation theory  this initial-value problem for each
$\lambda $ has a unique solution $\phi _{1}(x,\lambda )$. Moreover
[\cite{ti}, Teorem 7] this solution
is an entire function of $%
\lambda $ for each fixed $x\in \left[ -\pi,0\right).$ Using this
solutions we can prove that the equation (\ref{1.1}) on  the
right-hand interval $\in (0,\pi]$ of the considered interval
$\left[-\pi,0\right)\cup(0,\pi]$ has the solution
$u=\phi_2(x,\lambda)$ satisfying the initial conditions
\begin{eqnarray}\label{8}
&&y(0,\lambda) =\frac{1}{\rho_{12}}(\rho_{23}\phi_{1}(0,\lambda
)+\rho_{24}\frac{\partial\phi_{1}(0,\lambda )}{\partial x})\\
&& \label{9} y^{\prime }(0,\lambda)
=\frac{-1}{\rho_{12}}(\rho_{13}\phi_{1}(0,\lambda
)+\rho_{14}\frac{\partial\phi_{1}(0,\lambda )}{\partial x}).
\end{eqnarray}
By applying the method of \cite{os1}  we can prove that  the
equation (\ref{1.1}) on $ (0,\pi]$ has an unique solution
$\phi_{2}(x,\lambda)$ satisfying the conditions (\ref{8})-(\ref{9})
which also is an entire function of the parameter $\lambda $ for
each fixed $x \in (0,\pi]$. Consequently, the function  $\phi
(x,\lambda )$ defined by
\begin{equation}
\phi (x,\lambda )=\{
\begin{array}{c}
\phi _{1}(x,\lambda )\text{ \ for }x\in \lbrack -\pi,0) \\
\phi _{2}(x,\lambda )\text{ \ for }x\in (0,\pi].%
\end{array}
\label{(3.16)}
\end{equation}
 satisfies equation $(\ref{1.1})$,  the first boundary condition
$(\ref{1.4})$ and the both transmission conditions $(\ref{1.2})$ and $(%
\ref{1.3})$. Similarly, $\chi_{2}(x,\lambda)$   be solutions of
equation (\ref{1.1}) on the left-right interval $(0, \pi]$ subject
to initial conditions
\begin{equation}\label{4}
 \ \ y(\pi,\lambda)=-\sin\beta , \
\  \frac{\partial y(\pi,\lambda)}{\partial x}=\cos\beta.
\end{equation}
By virtue of [\cite{ti}, Teorem 7] each of these solutions are
entire functions of $\lambda$ for fixed x. By applying the same
technique we can prove  there is an unique solution
$\chi_{1}(x,\lambda)$ of equation (\ref{1.1}) the left-hand interval
$[-\pi, 0) \ $ satisfying the initial condition
\begin{eqnarray}\label{11}
&&y(0,\lambda) =\frac{-1}{\rho_{34}}(\rho_{14}\chi_{2}(0,\lambda
)+\rho_{24}\frac{\partial\chi_{2}(0,\lambda )}{\partial x}),\\
&& \label{12} y^{\prime }(0,\lambda)
=\frac{1}{\rho_{34}}(\rho_{13}\chi_{2}+\rho_{23}\frac{\partial\chi_{2}(0,\lambda
)}{\partial x}).
\end{eqnarray}
 By applying the similar technique as in \cite{os1} we
can prove that the solutions $\ \chi_{1}(x,\lambda)$ are also an
entire functions of parameter $\lambda$ for each fixed x.
Consequently, the function $ \chi(x,\lambda)$ defined  as
\begin{displaymath} \chi(x,\lambda)=\left
\{\begin{array}{ll}
\chi_{1}(x,\lambda), & x\in [-\pi, 0) \\
\chi_{2}(x,\lambda), & x\in (0, \pi] \\
\end{array}\right.
\end{displaymath}
satisfies the equation (\ref{1.1}) on whole $[-\pi, 0)\cup (0,
\pi]$, the other boundary condition (\ref{1.5}) and the both
transmission conditions (\ref{1.2}) and (\ref{1.3}).
 In the Hilbert Space
$\mathcal{H}=L_{2}[-1,0)  \oplus L_{2}(0,1] $ of two-component
vectors we define an inner product by
$$
<y,z>_{\mathcal{H}}:= \ \rho_{12} \int_{-\pi}^{0}
y(x)\overline{z(x)}dx + \rho_{34} \int_{0}^{\pi}
y(x)\overline{z(x)}dx
$$
for $y= y(x), \ z= z(x) \ \in \mathcal{H}$. We introduce the linear
operator $A:\mathcal{H}\rightarrow \mathcal{H}$  with domain of
definition satisfying the
following conditions\\
 i) $y \ \textrm{and} \ y'$ are absolutely
continuous in each of intervals $[-\pi,0)$ \ and $(0,\pi]$ and has a
finite limits $y(c\mp) \ \textrm{and} \ y'(c\mp)$\\
 ii) $\Gamma y(x) \in
\mathcal{H}, \ \Gamma_1 y(x)=\Gamma_2 y(x)=\Gamma_3 y(x)=\Gamma_4 y(x)=0,$\\
Obviously $D(A)$ is a linear subset dense in $\mathcal{H}$. We put
$$(Ay)(x)=\Gamma y(x), \ x\in \mathcal{H}
$$ for $y \in D(A)$. Then the problem $(\ref{1.1})-(\ref{1.5})$ is
equivalent to the equation
$$Ay=\lambda y $$ in the Hilbert space $\mathcal{H}$.
 Taking in view that the   Wronskians
  $ W(\phi_{i},\chi_{i};x) :=
\phi_{i}(x,\lambda)\chi'_{i}(x,\lambda)-\phi'_{i}(x,\lambda)\chi_{i}(x,\lambda)$
\ are independent of variable x  we shall denote  $
w_{i}(\lambda)=W(\phi_{i}, \chi_{i};x) \ (i=1,2)$. By using
(\ref{1.2}) and (\ref{1.3}) we have $ \label{7}
\rho_{12}w_{1}(\lambda)= \rho_{34}w_{2}(\lambda) $ for each $\lambda
\in \mathbb{C}$. It is convenient to introduce the notation
 \begin{equation}\label{kn}w(\lambda):=\rho_{34} w_{1}(\lambda) =  \rho_{12}
w_{2}(\lambda).\end{equation}
\begin{theorem}
For all $y, z \in D(A),$  the equality
\begin{eqnarray}\label{sim}<Ay,z>=<y,Az>
\end{eqnarray}
holds.
\end{theorem}
\begin{proof}
 Integrating by parts we have for all $y, z \in D(A),$
\begin{eqnarray}\label{2.z}
<Ay,z> &=& \rho_{12}\int_{-\pi}^{0} \Gamma y(x)\overline{z(x)}dx
+\rho_{34}\int_{0}^{\pi} \Gamma y(x)\overline{z(x)}dx
\nonumber\\
&=& \rho_{12}\int_{-\pi}^{0} y(x)\overline{\Gamma z(x)}dx
+\rho_{34}\int_{0}^{\pi} y(x)\overline{\Gamma z(x)}dx  +
\rho_{12}W[y, \overline{z};0-]\nonumber\\&-&\rho_{12}W[y,
\overline{z};-\pi] +\rho_{34} W[y,\overline{z};\pi] - \rho_{34} \
W[y,\overline{z};0]
\nonumber\\
&=&\ <y,Az>  + \rho_{12}W[y_{0}, \overline{z};0] - \rho_{12}
W[y_{0}, \overline{y_{0}};-\pi]\nonumber\\ &+&\rho_{34}
W[y_{0},\overline{y_{0}};\pi] - \rho_{34}\
W[y_{0},\overline{y_{0}};0]
\end{eqnarray}
From the boundary  conditions $(\ref{1.2})$-$(\ref{1.3})$ it is
follows obviously that
\begin{equation}\label{c1}W(y, \overline{z};-\pi)=0
 \ \textrm{and} \  W(y, \overline{z};\pi)=0
\end{equation}
The direct calculation gives
\begin{equation}\label{c3}\rho_{12}W(y,
\overline{z};0) = \rho_{34}W(y, \overline{z};0).
\end{equation}
Substituting (\ref{c1}) and (\ref{c3}) in (\ref{2.z}) we obtain the
equality (\ref{sim}).
\end{proof}
Relation (\ref{sim}) shows that the operator A is symmetric.
Therefore all eigenvalues of the operator A are real and two
eigenfunctions corresponding to the distinct eigenvalues are
orthogonal in the sense of the following equality
\begin{eqnarray}\label{2.3}\rho_{12}\int_{-\pi}^{0} y(x)z(x)dx + \rho_{34} \int_{0}^{\pi}
y(x)z(x)dx=0.
\end{eqnarray}
\begin{theorem}
 $D(A)$  is densely in the Hilbert space $\mathcal{H}$.
\end{theorem}
\begin{proof}
\end{proof}
\begin{theorem}
The operator A is self-adjoint.
\end{theorem}
\begin{proof}
\end{proof}

\section{The Green' s function}
Let us consider the inhomogeneous differential equation
\begin{eqnarray}\label{2.5}
y''+(\lambda-q(x)y=f(x), \ \ x \in[-\pi,0)\cup(0,\pi]
\end{eqnarray}
together with the  boundary conditions (\ref{1.4})-(\ref{1.5})
 and the transmission conditions (\ref{1.2})-(\ref{1.3}).
Making use of the definitions of the functions $ \phi_{i}, \chi_{i}$
(i = 1, 2) we see that the general solution of the differential
equation (\ref{2.5}) can be represented in the form
\begin{eqnarray}\label{2.10}
y(x,\lambda)=\left\{\begin{array}{c}
               \frac{\chi_{1}(x,\lambda)}{\omega_{1}(\lambda)}\int_{-\pi}^{x}\phi_{1}(\xi,\lambda)f(\xi)d\xi +
\frac{\phi_{1}(x,\lambda)}{\omega_{1}(\lambda)}\int_{x}^{0}\chi_{1}(\xi,\lambda)f(\xi)d\xi \\ +c_{1}\phi_{1}(x,\lambda)+d_{1}\chi_{1}(x,\lambda) \ , \ \ \ \ \ \ \ \ for \ x \in [-\pi,0)  \\
                \\
               \frac{\chi_{2}(x,\lambda)}{\omega_{2}(\lambda)}\int_{0}^{x}\phi_{2}(\xi,\lambda)f(\xi)d\xi +
\frac{\phi_{2}(x,\lambda)}{\omega_{2}(\lambda)}\int_{x}^{\pi}\chi_{2}(\xi,\lambda)f(\xi)d\xi\\ +c_{2}\phi_{2}(x,\lambda)+d_{2}\chi_{2}(x,\lambda) \ , \ \ \ \ \ \ \ \ for \ x \in (0,\pi] \\
             \end{array}\right.
\end{eqnarray}
where $c_{i}, d_{i} \ (i=1,2)$  are arbitrary constants. By
substitution into the boundary conditions (\ref{1.4})-(\ref{1.5}) we
have at once that
\begin{eqnarray}\label{2.11}d_{1}=0, \
c_{2}=0.
\end{eqnarray}
Further, substitution (\ref{2.10}) into transmission conditions
(\ref{1.2})-(\ref{1.3}) we have the inhomogeneous linear system of
equations for $c_{1}$ \textrm{and} $d_{1}$  , the determinant of
which is equal to $- \omega(\lambda)$ and  therefore is not vanish
by assumption. Solving that system we find
\begin{eqnarray}\label{2.12}
c_{1}=\frac{1}{\omega_{2}(\lambda)}\int_{0}^{\pi}\chi_{2}(\xi,\lambda)f(\xi)d\xi,
\ \ \ \
d_{2}=\frac{1}{\omega_{1}(\lambda)}\int_{-\pi}^{0}\phi_{1}(\xi
z,\lambda)f(\xi)d\xi
\end{eqnarray}
Putting this equations in (\ref{2.10}) we deduce that the problem
(\ref{2.5}), (\ref{1.2})-(\ref{1.4}) has an unique solution
$y:=y_{f}(x,\lambda)$ in the form
\begin{eqnarray}\label{2.14}
y_{f}(x,\lambda)=\left\{\begin{array}{c}
               \frac{\rho_{34}\chi_{1}(x,\lambda)}{\omega(\lambda)}\int_{-1}^{x}\phi_{1}(\xi,\lambda)f(\xi)d\xi +
\frac{\rho_{34}\phi_{1}(x,\lambda)}{\omega(\lambda)}\int_{x}^{0}\chi_{1}(\xi,\lambda)f(\xi)d\xi  \\
 +\frac{\rho_{12}\phi_{1}(x,\lambda)}{\omega(\lambda)}\int_{0}^{\pi}\chi_{2}(\xi,\lambda)f(\xi)d\xi \
  \ \ \ \ \ \ \ \ \textrm{for} \  x \in [-\pi,0) \\
                \\
\frac{\rho_{12}\chi_{2}(x,\lambda)}{\omega(\lambda)}\int_{0}^{x}\phi_{2}(\xi,\lambda)f(\xi)d\xi
+
\frac{\rho_{12}\phi_{2}(x,\lambda)}{\omega(\lambda)}\int_{x}^{1}\chi_{2}(\xi,\lambda)f(\xi)d\xi\\
 +\frac{\rho_{34}\chi_{2}(x,\lambda)}{\omega(\lambda)}\int_{-\pi}^{0}\phi_{1}(\xi,\lambda)f(\xi)d\xi\
  \ \ \ \ \ \ \ \ \textrm{for} \  x \in (0,\pi] \\
             \end{array}\right.
\end{eqnarray}
From this formula we derive that the Green' s function  of the
problem (\ref{2.5}), (\ref{1.2})-(\ref{1.4}) can be represented as
\begin{eqnarray}\label{2.15}
G(x,\xi;\lambda)=\left\{\begin{array}{c}
 \frac{\phi(\xi,\lambda)\chi(x,\lambda)}{\omega(\lambda)} \ \ \ \ \textrm{for} \  -\pi\leq \xi\leq x\leq \pi, \ \  x, \xi \neq 0\\
  \frac{\phi(x,\lambda)\chi(\xi,\lambda)}{\omega(\lambda)} \ \ \ \ \textrm{for} \  -\pi\leq x\leq \xi\leq \pi, \ \  x, \xi\neq 0 \\
             \end{array}\right.
\end{eqnarray}
and the formula (\ref{2.14}) can be rewritten in terms of this
Green' s function as
\begin{eqnarray}\label{2.16}
y_{f}(x,\lambda)&=& \rho_{12}\int_{-\pi}^{0}
G(x,\xi;\lambda)f(\xi)d\xi+ \rho_{34} \int_{0}^{\pi}
G(x,\xi;\lambda)f(\xi)d\xi
\end{eqnarray}
\begin{theorem}
 The resolvent operator $R(\lambda,A)$  is compact.
\end{theorem}
\begin{proof}
\end{proof}

\vskip 1truecm
\end{document}